\documentclass{amsart}
\usepackage{vmargin,amsmath,amsfonts,amssymb,amsthm,xcolor,amscd,latexsym,xspace,enumerate}
\usepackage[pagebackref, colorlinks=true,linkcolor=blue,citecolor=red]{hyperref}
\usepackage{hyperref}
\usepackage{pdfpages}
\usepackage{tikz-cd}
\usetikzlibrary{arrows}
\usetikzlibrary{intersections}
\usepackage{pgfplots}
\usetikzlibrary{calc,3d,shapes, pgfplots.external, intersections}
\usepackage{pgfplots}
\usepackage[all]{xy}
\usepackage{wrapfig}
\usetikzlibrary{matrix,arrows,arrows.meta}

\newtheorem{theorem}{Theorem}[section]
\newtheorem{proposition}[theorem]{Proposition}
\newtheorem{lemma}[theorem]{Lemma}
\newtheorem{corollary}[theorem]{Corollary}
\theoremstyle{definition}
\newtheorem{definition}[theorem]{Definition}
\newtheorem{example}[theorem]{Example}
\newtheorem{problem}[theorem]{Problem}
\newtheorem{remark}[theorem]{Remark}
\numberwithin{equation}{section}
\newcommand{\R}{\mathbb R}

\newcommand{\N}{\mathbb N}
\newcommand{\Z}{\mathbb Z}
\newcommand{\T}{\mathbb T}
\newcommand{\one}{\mathbf 1}
\newcommand{\Va}{\mathcal V_a}
\newcommand{\Pa}{\mathcal P_a}
\newcommand{\gk}{\mathfrak k}
\DeclareMathOperator{\Ad}{Ad}
\DeclareMathOperator{\diver}{div}
\DeclareMathOperator{\vol}{vol}

\DeclareMathOperator{\SU}{SU}
\DeclareMathOperator{\SO}{SO}

 \pgfplotsset{compat=1.18} 

\title[On the  isoperimetric profile in infinite dimensional pro-Lie groups]{On the    isoperimetric profile \\
in infinite dimensional pro-Lie groups}

\author[E. Kazeem]{Eniola Kazeem}
\address{Eniola Kazeem\endgraf
Department of Mathematical Sciences, Federal University of Technology \endgraf
Akure, PMB 704, Akure, Ondo State, Nigeria}

\author[F. G. Russo]{Francesco G. Russo}
\address{Francesco  G. Russo\endgraf
School of Science and Technology, University of Camerino\endgraf 
via Madonna delle Carceri 9, 62032 Camerino, Italy; \endgraf
and \endgraf
Department of Mathematics and Applied Mathematics, University of the Western Cape\endgraf 
Private Bag X17, 7535 Bellville, South Africa; \endgraf
and \endgraf
Department of Mathematics and Applied Mathematics, University of Cape Town\endgraf Private Bag X1, Rondebosch, 7701 Cape Town, South Africa}

\date{17th of October 2026}

\begin{document}

\begin{abstract}
Let $K$ be a nontrivial compact connected Lie group with a bi-invariant Riemannian metric. On the compact pro-Lie group $K^{\N}$ we define a weighted variation by summing the distributional variations in its Lie group coordinates. We prove that all sublevel sets with bounded $L^1$-norm and weighted variation are compact in $L^1$ if and only if the coordinate weights tend to infinity. The proof gives a quantitative estimate for approximation by Haar averages over normal subgroups, illustrating that it is possible to apply a  methodology which is typical of the geometric measure theory in a specific area of the topological group theory. The resulting perimeter is invariant under left and right translations and is preserved under pullback by quotient homomorphisms equipped with quotient metrics. Isoperimetric minimizers exist at every prescribed Haar probability. For the infinite torus, the compactness condition is equivalent to properness of a weighted length on its character group and to compactness of a Fourier multiplier's resolvent. Finally, we show that  locally compact noncompact groups require an additional condition of tightness, in order to be treated as an appropriate generalization of the compact case.
\end{abstract}

\maketitle

\section{Introduction and statement of the main results}\label{section1}

Compactness for sets of finite perimeter is a basic ingredient to understand the isoperimetric problem in  metric spaces. On a complete Riemannian manifold of finite volume, uniform bounds on perimeter give global compactness for characteristic functions; minimizing sequences of prescribed volume then admit minimizers. Many recent results study the loss of volume at infinity; see  \cite{nardulli1, ritore2, ritore1, narduru}. More precisely, let $(M,g)$ be a smooth Riemannian manifold with volume measure $\vol_g$, according to Chavel \cite[Chapter IX]{chavel}. For a measurable set $E$ and an open $U\subset M$, the classical perimeter is
\begin{equation}\label{classicalperimeter}
\mathcal P_g(E,U)=\sup\left\{\int_E\diver_gY \, d\vol_g:
Y\in C_c^1(U,TM),\ \|Y\|_\infty\leq1\right\},
\end{equation}
where $C^1_c(U, TM)$ denotes the space  of \textit{continuously differentiable vector fields Y with compact support} defined from $U$ to the \textit{tangent bundle} $TM$,  $Y$ is a  vector field with compact support contained in $U$ of \textit{divergence }$ \mathrm{div}_{g} \ Y$  and \textit{supremum norm} $\|Y\|_\infty$ upper bounded by one.   Finite perimeter sets  represent  fundamental notions in geometric measure theory and were introduced by Caccioppoli and De Giorgi (see \cite[Chapters 2 and 3]{afp}).  Briefly,  $\mathcal{P}_{g}(E)=\mathcal{P}_{g}(E, M)$ and say that $E$ has \textit{finite perimeter in } $U$ if \begin{equation}\mathrm{vol}_g(E \cap U) < \infty \ \ \mbox{and} \ \  \mathcal{P}_{g}(E, U)<\infty.\end{equation} 
\textit{The isoperimetric profile} of $M$ is the function $I_M  : v \in \ ]0, \mathrm{vol}_{g}(M)[ \ \longmapsto I_M(v) \in [0, + \infty[$ such that
\begin{equation}I_M(v)=\inf\left\{\mathcal{P}_g(E_o)  :  E_o \subseteq M, \  \mathrm{vol}_{g}(E_o)=v \right\},
\end{equation} where the infimum is taken over all relatively compact open set $E_o$ with smooth boundary (and  this applies when $E_o=E$ is a measurable set of smooth boundary). An \textit{isoperimetric region} of a complete Riemannian manifold $M$ of volume $v$ is a set $E_o \subseteq M$ such that $\mathrm{vol}_{g}(E_o)=v$ and $\mathcal{P}_g(E_o)=I_M(v)$. A \textit{minimizing sequence} of sets of volume $v$ is a sequence of sets $E_k$ of finite perimeter such that $\mathrm{vol}_{g}(E_k)=v$ for all $k \in \mathbb{N}$ and $\lim_{k \rightarrow \infty} \mathcal{P}_g(E_k)=I_M(v)$. The behaviour of minimizing sequences of fixed volume was investigated in  \cite{nardulli1, narduru, ritore2, ritore1}.

\begin{theorem}[Classical Finite Volume Compactness Theorem, see \cite{narduru, ritore2, ritore1}]\label{thm:classical}
Let $M$ be a complete Riemannian manifold of finite  volume. If measurable sets $E_k$ have uniformly bounded perimeters, then there is an $L^1(M)$-convergent subsequence  with  finite perimeter limit. For a minimizing sequence of fixed volume $v\in]0,\vol_g(M)[$, the limit has volume $v$, is minimizing, and the perimeters of $E_k$ converge to its perimeter.
\end{theorem}

The proof of Theorem \ref{thm:classical} combines local BV compactness and lower semicontinuity (see  \cite[Chapters 2, 3]{afp} for the notion of BV function and of compactness theorems involving BV functions) with the fact that $\vol_g(M\setminus C)$ tends to zero along a compact exhaustion $C$. Theorem \ref{thm:classical}  is retained as motivation and comparison, but it cannot be applied verbatim to  infinite products of compact Lie groups which we will investigate later on.

In fact, we  need a modified definition of the classical perimeter \eqref{classicalperimeter} to study  infinite products of topological groups. Note that all topological groups are Hausdorff in the present paper, following \cite{G1, heyer, prolie, hofmor}. The passage from \textit{Lie groups} to \textit{pro-Lie groups} introduces a different issue, see Appendix \ref{app:preserved-background} for feedback and definitions on pro-Lie groups (or \cite{prolie, kazeem1, kazeem2}). A countable projective limit of compact Lie groups is a compact metrizable group, but need not be a finite dimensional manifold. A compatible product metric does not itself supply a Riemannian tensor, a volume form or a divergence operator. This is what happens when we consider the torus group $\T=\R/\Z$, which is a compact connected abelian Lie group, and form the product of finitely many copies $\T^n$ of $\T$, or ask more, and consider the cartesian product $\T^\N$ of countably many copies of $\T$. In the first case we still have a compact connected abelian Lie group, while in the second case $\T^\N$ is a compact connected abelian pro-Lie group which is more sophisticated to study. We therefore use normalized Haar measure and define variation directly through smooth functions on finite Lie group quotients. Riemannian metrics will be used on the finite   dimensional factors only.

Throughout Theorem \ref{thm:intro} which will be proved following the philosophy of Theorem \ref{thm:classical}, we have
\begin{equation}\label{infinitedimensionaliegroup}
G=K^{\N},\qquad \mbox{and} \qquad m=m_K^{\otimes\N},
\end{equation}
where $K$ is a fixed nontrivial compact connected Lie group and $m$ is the product measure of the normalized Haar measure $m_K$ on $K$, see \cite{heyer, hofmor}. Concerning adjoint representations, we note that an $\Ad(K)$-invariant inner product on its Lie algebra determines a bi-invariant Riemannian metric, see \cite{hilgertneeb, ledonne}. If $\mathcal V_j$ denotes the \textit{distributional variation} in the $j$th group coordinate, we consider for any $f \in L^1(G,m;\mathbb{R})$ 
\begin{equation}\label{newperimeter}
\Va(f)=\sum_{j\geq1}a_j\mathcal V_j(f),\qquad
\Pa(E)=\Va(\one_E),\qquad 0<a_j<\infty,
\end{equation}
where  $E$ is a measurable set in $G$ with respect to $m$ and $\one_E$ the constant function to $1$ on $E$. The definitions and all conventions appear in Section~\ref{sec:variation}  later on.

\begin{theorem}[First Main Theorem]\label{thm:intro}
Given $G$ as per \eqref{infinitedimensionaliegroup} and $\Va(f), \ \Pa(E)$ as per \eqref{newperimeter}, the following two conditions are equivalent:
\begin{enumerate}
\item[{\rm (i).}] $a_j\to\infty$;
\item[{\rm (ii).}] For every two positive constants $A$ and $B$, the set
\[
\{f\in L^1(G,m;\R):\|f\|_1\leq B,\ \Va(f)\leq A\}
\]
is compact in $L^1(G,m;\R)$.
\end{enumerate}
Under these conditions, the isoperimetric constant of $\Pa(E)$ exists, that is, the infimum of $\Pa(E)$ over measurable sets of any prescribed mass $v\in]0,1[$ is attained.
\end{theorem}

From now on, we write briefly $L^1(G,m)$ instead of $L^1(G,m);\mathbb{R}$ since we deal with real valued functions on pro-Lie groups. Our second main result specializes to the case of the torus group $\mathbb{T}$, which is an interesting one dimensional compact abelian connected Lie group.  The \textit{character group} (also called \textit{Pontryagin dual}) of $\T^{\N}$ is $\widehat{\T^{\N}}=\Z^{(\N)}$ (see \cite[Chapters 1, 8]{hofmor}), which can be visualized with the finite   support integral sequences. Indeed, a character sends some basic identity neighborhood into an arc containing no nontrivial subgroups; the unrestricted tail subgroup must then lie in its kernel. Thus every character depends on finitely many coordinates and has the form
\[
\chi_k(x)=\exp\!\left(2\pi i\sum^\infty_{j=1}k_jx_j\right),\qquad k\in\Z^{(\N)}=\widehat{\T}^{(\N)}=\widehat{\T^{\N}}.
\]
For $f\in L^1(\T^\N,m)$, let $\widehat f(k)=\int_{\T^\N} f \ \overline{\chi_k}\,dm$, and define the \textit{weighted length} $\ell_a(k)=   {\underset{j=1} {\overset{\infty}\sum}} \ a_j \ |k_j|$.

\begin{theorem}[Second Main Theorem]\label{thm:dual}
For the infinite dimensional pro-Lie group $\T^\N$, the following three conditions are equivalent:
\begin{enumerate}
\item[{\rm (i).}] $a_j\to\infty$;
\item[{\rm (ii).}] Every set $\{k\in\Z^{(\N)}:\ell_a(k)\leq R\}$ is finite with  $R$ positive constant;
\item[{\rm (iii).}] The nonnegative selfadjoint Fourier multiplier
\[
\Lambda_a\chi_k=\ell_a(k)\chi_k
\]
has compact resolvent on $L^2(\T^{\N},m)$, that is, has spectrum with compact complement.
\end{enumerate}
These conditions are  equivalent to the compactness assertion of Theorem~\ref{thm:intro} {\rm (ii)}. Moreover for a real function $f\in L^1(\T^\N,m)$, we have that
\[
2\pi \ \ell_a(k)| \ \widehat f(k)| \ \leq \ \Va(f).
\]
\end{theorem}

The notion of Fourier multiplier, of Fourier transform and other classical notions of Fourier analysis can be found in \cite{rudin}. It should be pointed out that the main estimate of Theorem \ref{thm:intro} (and consequently also of Theorem \ref{thm:dual}) is explicit. Write $E_nf$ for Haar averaging over the closed normal subgroup $1^n\times K^{\{n+1,n+2,\ldots\}}$, and consider the geodesic distance $d_K(1,h)$ on $K$ where $h \in K$. Put
\[
c_K=\int_K d_K(1,h)\,dm_K(h) \qquad \mbox{and} \qquad
\alpha_n=\inf_{j>n}a_j.
\]
Whenever $\alpha_n>0$,
\begin{equation}\label{eq:introtail}
\|f-E_nf\|_1\leq c_K\sum_{j>n}\mathcal V_j(f)
\leq\frac{c_K}{\alpha_n}\Va(f).
\end{equation}
Here is clear that the finite   dimensional compactness handles the quotient $K^n$, and \eqref{eq:introtail} controls the remaining coordinates. Conversely, if the weights stay bounded along infinitely many coordinates, copies of a fixed domain placed in distinct coordinates yield a uniform family of characteristic functions with bounded perimeter. This is a qualitative justification for  \eqref{newperimeter}.

However we may go ahead and consider different families of pro-Lie groups which are locally compact. For a concrete noncompact extension of Theorems \ref{thm:intro} and \ref{thm:dual}, let $d\geq1$ and consider
\begin{equation}\label{compactlygenerated}
\Gamma=\R^d\times K^{\N} \qquad \mbox{with} \ \ d\mu=dx\,dm.
\end{equation}
 These groups are also  known as \textit{compactly generated locally compact groups}, see \cite{prolie, hofmor,  kazeem2}.  Define $\mathcal V_0$ by the ordinary Euclidean divergence and Euclidean norm on its vector  fields, and define $\mathcal V_j$ as before for $j \ge1$. Now adjust our previous definition, considering
\begin{equation}\label{adjustedfuncationals}
\mathcal W_a(f)=\mathcal V_0(f)+\sum_{j\geq1}a_j\mathcal V_j(f).
\end{equation}
These functionals are lower semicontinuous and invariant under the left and right translations of $\Gamma$ and allow us to have a significant generalization of Theorem \ref{thm:intro}.

\begin{theorem}[Third Main Theorem]\label{thm:tight}
Suppose $a_j\to\infty$, $\Gamma$ as in \eqref{compactlygenerated}, $\mathcal{W}_a(f)$ as in \eqref{adjustedfuncationals} and that $\mathcal F\subset L^1(\Gamma, \mu)$ satisfies
\[
\sup_{f\in\mathcal F}\bigl(\|f\|_1+\mathcal W_a(f)\bigr)<\infty.
\]
Then $\mathcal F$ is relatively compact in $L^1(\Gamma, \mu)$ if and only if it is uniformly tight in the Euclidean variable:
\[
\lim_{R\to\infty}\sup_{f\in\mathcal F}
\int_{\{|x|>R\}\times K^{\N}}|f|\,d\mu=0.
\]
\end{theorem}

While Section \ref{section1} presented the main results, Section \ref{section2} shows the relevance of the structure of topological group with respect to the classical approach in general topological spaces. Section \ref{sec:haar} recalls  facts which are probably known for Haar measures and projective limits, but this section, along with Appendix \ref{app:preserved-background}, contains background which can be useful for readers. Section \ref{sec:variation} deals with \eqref{newperimeter} and how this should be treated in terms of the crucial lemmas of the proof of Theorem \ref{thm:intro} later in Section \ref{proof1}. There is a separate criterion of isoperimetric comparison in Section \ref{isoperimetricomparison}, which shows the important role of quotients and projective limits in the constructions. The proof of Theorem \
\ref{thm:dual} is placed in Section \ref{proof2}, while that of Theorem \ref{thm:tight} in Section \ref{proof3}. We end with Section \ref{subfinsler}, illustrating an interesting relation with the theory of Carnot groups (see \cite{ledonne}) and this, along with the problems which originally emerged in our investigations, show once more how the condition of finite dimension force us to change drastically methods and techniques of investigation.

\section{Why the group structure matters and the classical  perimeter fails}\label{section2}

The compactness step using finite   dimensional approximation is an analytic principle with applications beyond groups. The results here carry additional algebraic content. Haar probability fixes the measure compatibly with homomorphisms; multiplication produces measures preserving translations and one-parameter subgroups; normal subgroups give quotient groups and conditional expectations; and invariant inner products make the variation equivariant under both left and right translations. This scenario is absent in topological spaces which are not topological groups.

These structures enter proved conclusions, rather than only the terminology. Proposition~\ref{prop:averaging} gives an equivariant conditional expectation for every compact normal subgroup. Theorem~\ref{thm:quotient} proves exact preservation of variation and perimeter under pullback by compact Lie quotient homomorphisms with their quotient metrics. Corollary~\ref{cor:profiles} compares isoperimetric profiles along those homomorphisms. In the abelian example, Theorem~\ref{thm:dual} expresses the compactness condition on the Pontryagin dual group. Neither a general continuous map between compact spaces nor local compactness alone supplies these conclusions.

Compactness of the group is also consequential: a nonzero Haar measure has finite total mass precisely when its group is compact. On noncompact groups translation invariance produces families escaping to infinity; Section~\ref{proof3} proves a tightness version on $\R^d\times K^{\N}$. Thus the role of compactness is not interchangeable with the existence of an arbitrary finite measure on a topological space. Haar Theorem \cite[Theorem 2.8]{hofmor}, Peter--Weyl Theorem \cite[Theorem 3.7]{hofmor} and general translation criteria for $L^p$-compactness are classical \cite{heyer,hofmor,krukowski}. The results below use these structures to formulate and prove a precise weighted perimeter criterion, its equivariance and its behavior under group quotients. The variation is part of the specified geometric data: it is not asserted to be determined by the underlying topology, or to equal a metric BV perimeter for every compatible metric. The main theorem concerns products of a fixed connected compact Lie group, not all pro-Lie groups.

The finite   dimensional geometry fits the sub-Finsler framework developed by Le Donne \cite{ledonne}. In fact,  Section~\ref{subfinsler} identifies the precise tangent norm whose divergence variation is our finite   stage functional, and compares our statements with the scenario of Carnot groups and limits of Carnot--Carath\'eodory spaces. The infinite coordinate estimate \eqref{eq:introtail} is additional to those geometric results. The sufficiency argument is an application of classical compactness principles once this estimate is proved, and isoperimetric existence is a direct method consequence.

It's also appropriate to observe why the classical notion of perimeter fails in our present context of study. We tried several times (and with different notions of measures and perimeters), getting very often to very strange constructions, which brought us  to the finite dimensional case at the end. First of all,  on a compact Lie group with normalized Riemannian Haar measure, the one-block variation agrees with normalized classical perimeter. On a finite product, the weighted block variation is anisotropic but equivalent to ordinary BV; in the infinite product the quantitative estimate for averaging is the essential extra ingredient. Somehow this approach seemed to be a meaningful way to interpret properly our original attempts.

\begin{remark}[Invariant and Noninvariant Finite Volume]
The volume of a left   invariant Riemannian metric on a Lie group is a left Haar measure. Proposition~\ref{prop:finitehaar} below implies that the finite total volume  forces compactness. A noncompact manifold, including a manifold underlying a Lie group, can carry a finite volume noninvariant Riemannian metric. That situation does not furnish finite Haar measure, and should not be confused with the invariant construction used in the present paper.
\end{remark}

\begin{remark}\label{rem:metric} Let's see better what a product metric does give and what does not give. If $d_j$ are complete compatible metrics, then
\[
d(x,y)=\sum^\infty_{j=1}\frac{1}{2^j}\min\{d_j(x_j,y_j),1\}
\]
is a complete compatible product metric, and its restriction to a closed projective limit is complete. Its local form can be an $\ell^1$-norm, so it need not be a Riemannian distance. Nor does it imply a numerical formula for a volume of the projective limit. For example, in the constant identity system of circles of circumference one, take $d_j=d_0/2^j$. On the diagonal projective limit the displayed metric is $d_0/3$, of Riemannian length $1/3$, whereas the factor lengths tend to zero. Haar probability remains consistently of mass one. This separates metric construction from measure normalization.
\end{remark}

\section{Compact groups, projective limits and Haar averaging}\label{sec:haar}
 A pro-Lie group is a topological group isomorphic to a closed subgroup of a product of finite dimensional real Lie groups; equivalently it is an appropriate projective limit of Lie groups \cite{prolie}. In particular compact groups are pro-Lie groups. We use strict compact projective systems: bonding homomorphisms and limit projections are surjective.

For a directed set $I$, continuous homomorphisms $q_{ji}:G_i\to G_j$ for $i\geq j$ satisfy $q_{ii}=\operatorname{id}$ and $q_{kj}q_{ji}=q_{ki}$ whenever $i\geq j\geq k$. Their projective limit is
\[
 \varprojlim_{i \in I} G_i=\left\{(x_i)\in\prod_{i\in I} G_i : q_{ji}(x_i)=x_j\ (i\geq j)\right\}.
\]
For compact Hausdorff factors this is a closed compact subgroup of the product. If the bonding maps are onto,  it turns out that every projection $p_i: \underset{i \in I}\varprojlim \ G_i \to G_i$ is onto. The identities are $p_j=q_{ji}p_i$ and note that the pullback of functions has the reverse direction.

\begin{proposition}\label{prop:measure}
Let  $G_i$ be compact groups with surjective bonding homomorphisms such that $G=\underset {i \in I} \varprojlim \ G_i$. If finite positive Radon measures $\mu_i$ satisfy $(q_{ji})_*\mu_i=\mu_j$, there is a unique finite positive Radon measure $\mu$ on $G$ with $(p_i)_*\mu=\mu_i$. For normalized Haar measures the extension is normalized Haar measure. 
\end{proposition}
\begin{proof}
All factor masses equal a finite number $M$. On the algebra $\mathcal A=\underset{i \in I}\bigcup p_i^*C(G_i,\R)$ define
\[
L(f\circ p_i)=\int_{G_i}f\,d\mu_i.
\]
Lifting two representations to a common upper stage proves that $L$ is well defined and that $\mathcal A$ is an algebra. It contains constants and separates points, so Stone--Weierstrass makes it uniformly dense in $C(G,\R)$. Positivity and $|L(F)|\leq M\|F\|_\infty$ extend $L$ to a positive continuous functional on $C(G,\R)$. Riesz representation proves existence and uniqueness. If the factor measures are normalized Haar, the functional is invariant under left translations on cylinder functions, hence on every continuous function. This proves the last assertion.
\end{proof}

Projective limits and Haar measure constructions also occur in probabilistic questions about compact and locally compact groups; see \cite{kazeem1,kazeem2}. Proposition~\ref{prop:measure} supplies the precise statement needed for our investigations. Write $L_gf(x)=f(g^{-1}x)$ and $R_gf(x)=f(xg)$ and note that compact groups are unimodular, so both operators $L_g$ and $R_g$ are isometries on $L^p(G,m_G)$, where $1\leq p\leq\infty$  and $m_G$ denotes the Haar measure on $G$, see \cite{heyer, hilgertneeb,  prolie, hofmor} for details.

\begin{proposition}\label{prop:averaging}
Let $N$ be a closed normal subgroup of a compact group $G$, and $q:G\to G/N$. Then $q_*m_G=m_{G/N}$ and
\[
E_Nf(x)=\int_N f(xh)\,dm_N(h)
\]
is an orthogonal    projection on $L^p(G,m_G)$ of norm one   (with $1\leq p<\infty$) onto the functions pulled back from $G/N$. The formula is defined first on continuous functions and then extended by $L^p$ contraction. It is the conditional expectation onto the quotient sigma-algebra and commutes with every $L_g$ and $R_g$.
\end{proposition}
\begin{proof}
The pushforward is a left   invariant probability, hence Haar. Fubini and Haar invariance give $E_N^2=E_N$, positivity, contraction and preservation of integrals against bounded functions constant on cosets. These are the conditional-expectation identities. Commutation with $L_g$ is immediate. For $R_g$, substitute $h\mapsto g^{-1}hg$; normality keeps the integration inside $N$, and its normalized Haar measure is invariant under this automorphism. These arguments also identify the range with quotient functions, modulo null sets.
\end{proof}

\begin{proposition}\label{prop:representations}
Let $\pi:G\to U(V)$ be a finite   dimensional irreducible unitary representation of a compact group with $V$ finite dimensional complex (or real) vector space and  $N$  closed normal subgroup of $G$. On every matrix coefficient of $\pi$, the operator $E_N$ is the identity if $N\subseteq\ker\pi$, and is zero otherwise.
\end{proposition}
\begin{proof}
The operator $P_N=\int_N\pi(h)\,dm_N(h)$ is the orthogonal projection onto the $N$-fixed vectors. Normality makes this subspace $G$-invariant. Irreducibility implies that it is either $V$ or zero. For $c(x)=\langle\pi(x)v,w\rangle$ one has $E_Nc(x)=\langle\pi(x)P_Nv,w\rangle$, proving the assertion.
\end{proof}

Thus quotient averaging retains precisely the irreducible representations factoring through the quotient. This is a standard consequence of representation theory, and gives an algebraic interpretation of the approximation in Lemma \ref{thm:tower}. The regular actions on $L^2(G,m_G)$ are unitary and this Hilbert space is a space of functions on $G$, but unfortunately its existence supplies no Hilbert manifold structure on $G$.

\begin{lemma}\label{thm:tower}
Let $G$ be compact metrizable group with Haar measure $m_G$ and  $N_1\supseteq N_2\supseteq\cdots$  closed normal subgroups such that $\underset{n \in \N} \bigcap N_n=1$.   Then 
\[
E_n f(x):= E_{N_n}f(x)=\int_{N_n} f(xh)\,dm_{N_n}(h)
\]
converges to $f$ in $L^1(G,m_G)$. Moreover, a family $\mathcal F\subset L^1(G,m_G)$ is relatively compact if and only if the following two conditions are satisfied:
\begin{enumerate}
\item[{\rm (i).}] $E_n\mathcal F$ is relatively compact for each $n$;
\item[{\rm (ii).}] $\sup_{f\in\mathcal F}\|f-E_nf\|_1\to0$.
\end{enumerate}
All approximations commute with left and right translations.
\end{lemma}
\begin{proof}
The groups $N_n$ are eventually contained in each identity neighborhood: otherwise nested compact sets $N_n\setminus U$ would have nonempty intersection. Uniform continuity gives $E_nf\to f$ uniformly for continuous $f$. Density and contraction give strong $L^1$ convergence. On a compact subset of $L^1$ this convergence is uniform, by a finite net argument and the bound $\|I-E_n\|\leq2$. This proves necessity of (ii); continuity proves (i). Conversely, first approximate $\mathcal F$ uniformly by $E_n\mathcal F$ and then by a finite net in that projected family. Thus $\mathcal F$ is totally bounded in the complete space $L^1$. Equivariance follows from Proposition~\ref{prop:averaging}.
\end{proof}

Peter--Weyl Theorem supplies such a tower with compact Lie quotients for every compact metrizable group. Indeed, finite dimensional continuous unitary representations separate points. Given a countable identity neighborhood basis $U_n$, compactness of $G\setminus U_n$ gives finitely many representations whose kernels have intersection inside $U_n$. Finite intersections of all kernels chosen up to stage $n$ give $N_n$. The quotient is a closed subgroup of a finite product of unitary groups, hence a compact Lie group. The separating representation input is classical; see \cite{hilgertneeb, hofmor, kubisturek}. Lemma \ref{thm:tower} alone does not estimate its approximation error: \eqref{eq:introtail} supplies that estimate for the product groups considered here.

\section{Weighted variation from invariant derivatives}\label{sec:variation}

Fix a nontrivial compact connected Lie group $K$ and an $\Ad(K)$-invariant inner product on the associated Lie algebra $\gk$. Such an inner product is obtained by averaging any inner product under the adjoint action. It defines a bi-invariant Riemannian metric $g_K$ with geodesic distance $d_K$. Its Riemannian volume, divided by $\vol_{g_K}(K)$, is $m_K$.

Let $X_1,\ldots,X_r$ be an orthonormal basis of $\gk$. On $G=K^{\N}$ let $X_{j,s}$ differentiate the $j$th coordinate by $x_j\mapsto x_j\exp(tX_s)$. A smooth cylinder function depends smoothly on finitely many coordinates; a cylinder vector field $\Phi=(\Phi_1,\ldots,\Phi_r)$ is an $\gk$-valued such function. Put
\[
\diver_j\Phi=\sum_{s=1}^rX_{j,s}\Phi_s,
\qquad |\Phi(x)|^2=\sum^r_{s=1}|\Phi_s(x)|^2.
\]
The one-parameter flows preserve Haar measure, so each $X_{j,s}$ has zero divergence relative to $m$ and integration by parts holds for cylinder functions.

\begin{definition}[The Weighted Perimeter on Compact Connected Groups]\label{def:variation} Given $G$ as per \eqref{infinitedimensionaliegroup}
and  $f\in L^1(G,m)$, we define
\begin{align*}
\mathcal V_j(f)&=\sup\left\{\left|\int_G f\,\diver_j\Phi\,dm\right|:
\Phi\text{ smooth cylinder},\ \|\Phi\|_\infty\leq1\right\},\\
\Va(f)&=\sum_{j\geq1}a_j\mathcal V_j(f),\qquad
\Pa(E)=\Va(\one_E),\qquad 0<a_j<\infty.
\end{align*}
Values in $[0,\infty]$ are allowed, and measurable sets are identified modulo $m$-null sets.
\end{definition}

Of course, Definition \ref{def:variation} formalizes \eqref{newperimeter}, using the notion of distributional derivative which is typical of the theory of BV-functions and of the geometric measure theory, see \cite{afp}. Note also that Definition \ref{def:variation} is independent of the chosen orthonormal basis. For a smooth cylinder $f$, integration by parts and smooth approximation of the normalized gradient give
\[
\mathcal V_j(f)=\int_G\left(\sum^r_{s=1}|X_{j,s}f|^2\right)^{1/2}dm.
\]
At a finite stage, $\Va$ is the sum of weighted block variations. Even there it generally differs from the usual product Riemannian variation, which combines all block gradients by a single Euclidean norm (its precise Finsler interpretation is given in Proposition~\ref{prop:finsler}). No Riemannian manifold structure on the infinite product is used. In fact, we may encounter difficulties to adapt the Alexandroff-Mazurkiewicz metric already to $\T^\N$ (see Appendix \ref{app:preserved-background} and  \cite[Theorem 24.11]{willard}) if we do not consider a notion of perimeter of the type of  Definition \ref{def:variation}.

\begin{proposition}[Equivariance and Lower Semicontinuity]\label{prop:equivariant}
For $g\in G=K^\N$ with $K$ compact connected Lie group and $f\in L^1(G,m)$,
\[
\mathcal V_j(L_gf)=\mathcal V_j(R_gf)=\mathcal V_j(f).
\]
Consequently $\Va$ and $\Pa$ are invariant under left and right translations. Every $\mathcal V_j$ and $\Va$ is lower semicontinuous in $L^1$. If $A_j$ averages over coordinate $j$ and $E_n$ over all coordinates after $n$, then
\[
\mathcal V_i(A_jf)\leq\mathcal V_i(f),\qquad
\mathcal V_i(E_nf)\leq\mathcal V_i(f),\qquad
\mathcal V_i(E_nf)=0\quad(i>n).
\]
\end{proposition}
\begin{proof}
Left translations commute with the left   invariant frame. Right translation in coordinate $j$ transforms that frame by $\Ad(g_j^{-1})$, an orthogonal map. The corresponding transformations of fields preserve their pointwise norms and the Haar integrals. Applying the inverse translation proves equality. Averaging the transformed fields gives the variation inequalities, and integration of a derivative over the averaged coordinate gives zero. Finally each fixed integral is continuous on $L^1$; taking its supremum proves lower semicontinuity. The functional $\Va$ is the supremum of its lower semicontinuous finite partial sums.
\end{proof}

\begin{lemma}[Translations in One Coordinate]\label{lem:translate}
For $h\in K$ with $K$ compact connected  Lie group, let $R_{j,h}$ translate the $j$th coordinate. Then
the geodesic distance $d_K(1,h)$  implies  the bound \[
\|R_{j,h}f-f\|_1\leq d_K(1,h) \  \mathcal V_j(f).
\]
In particular, if $A_j$ averages over coordinate $j$, then
\begin{equation}\label{eq:oneaverage}
\|f-A_jf\|_1\leq c_K\mathcal V_j(f),\qquad \mbox{where} \qquad
c_K=\int_Kd_K(1,h)\,dm_K(h).
\end{equation}
\end{lemma}
\begin{proof}
The bi-invariant metric is complete; a minimizing geodesic from $1$ to $h\in K$ has the form $t\mapsto\exp(tX)$, $0\leq t\leq1$, with $|X|=d_K(1,h)$. For a smooth scalar cylinder $\varphi$ with $\|\varphi\|_\infty\leq1$, change variables using Haar measure and integrate the derivative of $\varphi(x_j\exp(-tX))$. The resulting check of Definition~\ref{def:variation} has vector coefficient $X\varphi/|X|$ and norm at most one. The integral is bounded by $|X|\mathcal V_j(f)$. Smooth scalar cylinder functions of norm at most one in $L^1$-norm. For clarity, approximate a real $u\in L^1$ by a smooth cylinder $w$, and applies with $w/(w^2+\varepsilon^2)^{1/2}$; the integral tends to $\|u\|_1$ as the two approximation errors vanish. Taking this supremum proves the translation inequality. Averaging it in $h$ and using \[A_jf=\int_KR_{j,h}f\,dm_K(h)\] proves \eqref{eq:oneaverage}.
\end{proof}

\begin{lemma}[Uniform Control of Later Coordinates]\label{lem:tail}
For every $f\in L^1(G,m)$ with $G=K^\N$ and $K$ compact connected Lie group, we have that
\[
\|f-E_nf\|_1\leq c_K\sum_{j>n}\mathcal V_j(f).
\]
In particular, if $\alpha_n=\inf_{j>n}a_j>0$, then the right side is at most $c_K\Va(f)/\alpha_n$.
\end{lemma}
\begin{proof}
The commuting contractions $A_j$ and \eqref{eq:oneaverage} give
\[
\|f-A_{n+1}\cdots A_Nf\|_1\leq c_K\sum_{j=n+1}^{N}\mathcal V_j(f).
\]
As $N\to\infty$, these averages converge to $E_nf$ in $L^1$-norm. This is eventually an equality for a function depending on finitely many coordinates, and holds generally by density and contraction. Pass to the limit and use $a_j\geq\alpha_n$ for $j>n$.
\end{proof}

\section{Proof of the first main theorem}\label{proof1}

In the present section we prove Theorem \ref{thm:intro} and the following lemma is crucial.

\begin{lemma}\label{lem:finite}
Consider $K$ compact connected Lie group. For a fixed positive integer $n$, a family of functions on $K^n$ with bounded $L^1$ norm and bounded coordinate variations $\mathcal V_1,\ldots,\mathcal V_n$ is relatively compact in $L^1(K^n,m_{K^n})$.
\end{lemma}
\begin{proof}
Telescoping translations and Lemma~\ref{lem:translate} give
\[
\|R_hu-u\|_1\leq\sum_{j=1}^n d_K(e,h_j)\mathcal V_j(u).
\]
Convolving with a smooth approximate identity on the compact Lie group $K^n$ approximates the family uniformly in $L^1$. At each fixed smoothing scale, its $L^1$-norm bound  constrain the smoothed functions and their derivatives uniformly. Arzel\`a--Ascoli Theorem (see \cite{afp} or \cite[Theorem 43.15, Problems 43 E]{willard}) makes the smoothed family precompact, and approximation gives the conclusion. This is the finite   dimensional translation argument of the  compactness principle; see \cite{afp,krukowski}.
\end{proof}

\begin{proof}[Proof of Theorem~\ref{thm:intro}: compactness]
Suppose $a_j\to\infty$, so $\alpha_n\to\infty$. Lemma~\ref{lem:tail} gives uniform approximation of the indicated sublevel set $\mathcal F_{A,B}$ by $E_n\mathcal F_{A,B}$. At a fixed stage,
\[
\|E_nf\|_1\leq B,\qquad
\mathcal V_j(E_nf)\leq A/a_j\quad(1\leq j\leq n).
\]
Lemmas \ref{lem:finite} and \ref{thm:tower} imply precompactness, but this is sufficient to conclude compactness. In fact, the lower semicontinuity makes $\mathcal F_{A,B}$ closed, hence compact.

Conversely, if $a_j$ does not diverge, choose distinct $j_k$ with $a_{j_k}\leq C<\infty$. A nontrivial small smooth domain $D\subset K$ has $0<b=m_K(D)<1$ and finite normalized Riemannian perimeter $p>0$. For
\[
F_k=\{x:x_{j_k}\in D\}
\]
we have $m(F_k)=b$, $\Pa(F_k)=a_{j_k}p\leq Cp$, while we also have 
\[
\|\one_{F_k}-\one_{F_\ell}\|_1=2b(1-b)>0\quad(k\ne\ell).
\]
Thus the sublevel set with $A=Cp$, $B=b$ is not relatively compact. This establishes the equivalence. Note that we may chopse two arbitrary constants $A,B$ and this allows us to include degenerate cases without claiming that every individual sublevel set must fail compactness when the weights do not diverge.
\end{proof}

Now we go ahead to prove the final part of Theorem \ref{thm:intro}, where compactness is used just to describe the isoperimetric profile and quantitative bounds on the perimeter.

\begin{lemma}[Behaviour of Competitors]\label{lem:volume}
For every $v\in]0,1[$ there exists a finite perimeter measurable set $D_v\subset K$ with $m_K(D_v)=v$, where $K$ is compact connected Lie group of normalized Haar measure $m_K$.
\end{lemma}
\begin{proof}
Cover $K$ by finitely many open coordinate boxes $U_1,\ldots,U_s$ whose closures lie in coordinate charts. Each box has piecewise smooth boundary and finite perimeter. Disjoint them two by two by $D_i=U_i\setminus\bigcup_{k<i}U_k$. Finite unions, intersections and differences of finite   perimeter sets have finite perimeter (see \cite{afp} since this is a general property of BV functions and BV product inequality). Therefore $D_i$ makes a partition of $K$ up to null boundaries.

Take the first index for which the cumulative mass reaches $v$. Fill the preceding pieces completely and cut the remaining piece $D_i$ by a sublevel of the first coordinate in its chart. Its mass as a function of the cutting parameter is continuous, by dominated convergence, since coordinate hyperplanes have zero volume. It ranges from zero to $m_K(D_i)$. Every such cut of the coordinate box has piecewise smooth boundary; intersecting with $D_i$ keeps finite perimeter. The Intermediate Value Theorem (see this classical fact in \cite{rudin}) supplies the required mass. Finally, one can check that the construction is complete and the result follows.
\end{proof}

Concerning the final part of the proof of Theorem \ref{thm:intro}, we must deal with isoperimetric inequalities and existence of minimizers. Therefore we prove the final part of  Theorem \ref{thm:intro} in terms of Theorem \ref{thm:isoperimetry} and Corollary \ref{cor:poincare} below. This helps to identify independently a classical idea of geometric measure theory when we deal with theorems of regularity and compactness, namely the problem of existence of minimizers when one has isoperimetric inequalities.

\begin{theorem}[Existence of Minimizers]\label{thm:isoperimetry}
Assume $a_j\to\infty$ and define the isoperimetric profile  depending on the costant $a$ and on the compact connected Lie group $K$ of normalized Haar measure $m_K$ and product measure $m=m^{\otimes \N}_K$ by
\[
I_{K,a}(v)=\inf\{\Pa(E):m(E)=v\},\qquad 0<v<1.
\]
The infimum is finite and attained. Every minimizing sequence has an $L^1$-convergent subsequence of characteristic functions, whose limit is a minimizer of the same mass; its perimeters converge to the limiting perimeter. All left and right translates of a minimizer are minimizers.
\end{theorem}
\begin{proof}
The cylinder $\{x:x_1\in D_v\}$ supplied by Lemma~\ref{lem:volume} has finite perimeter, so the infimum is finite. A minimizing sequence has bounded variation and $L^1$-norm equal to $v$. Compactness gives an $L^1$-limit, and a further subsequence converges almost everywhere. The limit is a characteristic function; convergence of integrals preserves its mass. Lower semicontinuity proves minimality. The perimeter values converge by the definition of a minimizing sequence. Translation invariance of Haar measure and Proposition~\ref{prop:equivariant} prove the final assertion.
\end{proof}

The conclusion is existence and compactness, with no assertion of smooth boundaries, uniqueness or a classification of minimizers. Since $G=K^\N$ is a compact group in Theorem \ref{thm:intro}, we have compactness after arbitrary choices of translations of a uniformly bounded perimeter sequence.

\begin{corollary}[Haar--Poincar\'e Estimate and Positive Isoperimetric Profile]\label{cor:poincare}
Consider $K$ compact connected Lie group and $G=K^\N$ as per \eqref{infinitedimensionaliegroup} and \eqref{newperimeter}. If $a_*:=\inf_ja_j>0$, then
\[
\left\|f-\int_G f\,dm\right\|_1\leq\frac{c_K}{a_*}\Va(f).
\]
Consequently for $v \in ]0,1[$ we have
\[
I_{K,a}(v)\geq\frac{2a_*}{c_K} \ v \ (1-v) \ > \qquad and 
\qquad I_{K,a}(v)=I_{K,a}(1-v).
\]
In particular these conclusions hold when $a_j\to\infty$.
\end{corollary}
\begin{proof}
Average the first $N$ coordinates and telescope \eqref{eq:oneaverage}. By cylinder density the averages tend in $L^1$ to the Haar mean, giving \[\left\| f - \int_G f dm\right\|_1\leq c_K\sum^\infty_{j=1}\mathcal V_j(f).\] For an indicator of mass $v$, the left side is $2v(1-v)$. Finally $\int_G\diver_j\Phi\,dm=0$, so a set and its complement have the same variation and perimeter.
\end{proof}

\section{Quotient homomorphisms and perimeter}\label{isoperimetricomparison}

The next result depends on a normal subgroup and its quotient Lie algebra, not on a general continuous map of compact spaces.

\begin{theorem}[Exact Preservation under Quotient Pullback]\label{thm:quotient}
Let $q:K\to H$ be a surjective homomorphism of compact connected Lie groups. Equip $K$ with a bi-invariant metric and $H$ with its quotient metric, so $dq$ restricts to an isometry from $(\ker dq)^\perp$ onto $\mathfrak h$. Give both groups normalized Haar measure. For $Q=q^{\N}$ and real $f\in L^1(H^{\N},m^{\otimes \N}_H)$, we have
\[
\|f\circ Q\|_1=\|f\|_1,\qquad
\mathcal V_j^K(f\circ Q)=\mathcal V_j^H(f),\qquad
\mathcal V_a^K(f\circ Q)=\mathcal V_a^H(f).
\]
Then Haar mass and weighted perimeter are preserved under inverse images of measurable sets.
\end{theorem}
\begin{proof}
Write $N=\ker q$ and split $\gk=\mathfrak n\oplus\mathfrak n^\perp$ orthogonally. Both summands are $\Ad(K)$-invariant, and the second maps isometrically onto $\mathfrak h$. Quotient Haar measure is the pushforward of $m_K$, proving the norm identity. Lifting a  field on $H^{\N}$ horizontally gives an admissible  field on $K^{\N}$ with the same norm and with divergence pulled back from $H^{\N}$. This proves  $\mathcal V_j^K(f\circ Q)\geq\mathcal V_j^H(f)$.

For the reverse inequality, start with a cylinder field $\Phi$ on $K^{\N}$. Average its coefficients along the fibres of $Q$ and project to $\mathfrak n^\perp$. The resulting field descends to a smooth cylinder field on $H^{\N}$ and has norm at most one. For a vertical basis vector, the fibre average of its derivative is zero, by invariance of Haar measure on $N$. For a horizontal basis vector, differentiation commutes with fibre averaging and descends to the corresponding derivative on $H$.

To justify the latter assertion explicitly, if $Z\in\mathfrak n^\perp$ and $h\in N$, then $\Ad(h)Z=Z$: normality preserves the horizontal space and $dq(\Ad(h)Z)=dq(Z)$, on which $dq$ is injective. Thus $\exp(tZ)$ commutes with each fibre translation. These identities apply to the finitely many coordinates of the  field, including disconnected kernels. Haar disintegration and Fubini-Tonelli Theorem (see \cite{afp, heyer, rudin}) now identify the original integral against $f\circ Q$ with the descended  integral against $f$. Taking suprema proves the reverse inequality and summing gives the weighted identity.
\end{proof}

\begin{corollary}[Criterion of Isoperimetric Comparison]\label{cor:profiles}
Under the hypotheses of Theorem~\ref{thm:quotient}, assume $H$ is nontrivial and $a_j\to\infty$. Then for $v \in ]0,1[$ we have
\[
I_{K,a}(v)\leq I_{H,a}(v).
\]
Moreover the infimum on $K^{\N}$ restricted to sets of the form $Q^{-1}(E)$ equals $I_{H,a}(v)$ and is attained.
\end{corollary}
\begin{proof}
Every quotient competitor lifts with the same mass and perimeter. Lifting a quotient minimizer proves both assertions. Note that  there is no claim that every unrestricted minimizer is saturated by $\ker Q$.
\end{proof}

\begin{problem}The quotient metric hypothesis is essential in Corollary \ref{cor:profiles}; arbitrary metrics on specific compact Lie groups  need not identify horizontal lengths. Similarly, the weights must be the same on both products. Together with bilateral translation invariance, this gives the framework to go ahead  in terms of category theory and explore functorial and equivariant group theoretic contents in the present perspective of investigation.
\end{problem}

\section{Proof of the second main theorem and some nonabelian examples}\label{proof2}
We proceed to prove Theorem \ref{thm:dual}, but it's appropriate  to note that compactness may be insufficient for a result of the type of Theorem \ref{thm:dual} (or of type of Theorem \ref{thm:intro}). For $K=\T$, we have $X=\partial_x$ and
\[
c_{\T}=\int_0^1\min(t,1-t)\,dt=\tfrac14.
\]
Thus \eqref{eq:introtail} has the constant $1/4$. It is optimal in the one coordinate averaging estimate: note that  $f=\one_{[0,1/2[}$ satisfies \[\left\|f-\int_K f dm_K\right\|_1=1/2\qquad \mbox{and} \qquad \mathcal V(f)=2.\] A cylinder $\{x:x_j\in[0,v[\}$ has mass $v$ and perimeter $2a_j$. When $a_j=1$, half circle cylinders in distinct coordinates have pairwise $L^1$-distance $1/2$ and perimeter $2$. Thus compactness of the underlying group alone is insufficient.

\begin{proof}[Proof of Theorem \ref{thm:dual}]
If (i) holds, a vector of length at most $R$ can use only coordinates with $a_j\leq R$, of which there are finitely many, and each coefficient satisfies $|k_j|\leq R/a_j$. Conversely a bounded subsequence of weights gives infinitely many distinct coordinate characters of bounded length. This proves (i) equivalent to (ii).

The characters form an orthonormal basis, by finite   dimensional Fourier theory and density of cylinders. Define $\Lambda_a$ on those square summable functions $u \in L^2(\T^\N,m)$ for which \[\sum^\infty_{k=1} \ell_a(k)^2 \ |\widehat u(k)|^2<\infty.\] The diagonal operator $(I+\Lambda_a)^{-1}$ has eigenvalues $(1+\ell_a(k))^{-1}$; finite   rank truncation proves compactness exactly when every bounded length set is finite. Finally real trigonometric functions with a chosen phase give \[2\pi \ \ |k_j| \ \ |\widehat f(k)|\leq\mathcal V_j(f).\] Multiply by $a_j$ and sum over the support of $k$. The result follows.
\end{proof}

The above proof deals with a spectral characterization on the dual group, not an identification of the $L^1$ variation with the graph norm of $\Lambda_a$. Its usefulness here is that the same weight condition controls group characters and weighted BV compactness. This behaviour is peculiar of tori, due to the specific structure of the character group.

\begin{remark} We want to note why the connected factors provide a real restriction in Theorem \ref{thm:dual}. If a compact Lie factor $K$ is disconnected, choose a nontrivial union $D$ of its connected components. All Lie derivatives vanish, while $b=m_K(D)$ lies in $]0,1[$. The cylinders $\{x:x_j\in D\}$ then have zero weighted perimeter for every choice of weights, but pairwise $L^1$-distance $2b(1-b)$. Thus connectedness cannot simply be dropped from Theorems \ref{thm:intro} and \ref{thm:dual}. More generally, in a totally disconnected group every continuous homomorphism $\R\to G$ is trivial. A variation defined solely from continuous one-parameter subgroups vanishes on an infinite profinite group. On $\Z(2)^{\N}$ the coordinate half spaces again form a separated family. Therefore profinite groups should be studied through different analytic notions. 
\end{remark}

Let' see  what happens to nonabelian  compact groups.

\begin{remark} \label{su2} The nonabelian group $\SU(2)^{\N}$ to the countable power provides an interesting framework of application of Theorem \ref{thm:classical} and differs from that of Theorem \ref{thm:dual}. Identify $\SU(2)$ with the unit quaternions, carrying the round metric on $\mathbb{S}^3$. Quaternion multiplication acts by orthogonal maps on both sides, so this metric is bi-invariant. The normalized radial density about the identity is $(2/\pi)\sin^2r\,dr$, $0\leq r\leq\pi$. Hence
\[
c_{\SU(2)}=\frac2\pi\int_0^\pi r\sin^2r\,dr=\frac\pi2.
\]
The geodesic cap $D_r$ has
\[
m_K(D_r)=\frac{r-\sin r\cos r}{\pi},\qquad
\mathcal P_K(D_r)=\frac{2\sin^2r}{\pi}\quad(0<r<\pi).
\]
The cylinder $\{x:x_j\in D_r\}$ therefore has weighted perimeter $2a_j\sin^2r/\pi$. These are explicit competitors at every volume, without asserting that they are minimizers in the infinite product. For $a_j\to\infty$, Theorem~\ref{thm:intro} applies to this nonabelian compact group. The central quotient
\[
\SU(2)\longrightarrow\SO(3)=\SU(2)/\{\pm I_2\}
\]
with the quotient metric gives an instance of Theorem~\ref{thm:quotient}; the corresponding profile comparison is instructive as application of Corollary \ref{cor:profiles}.
\end{remark}

\section{Proof of the third main theorem}\label{proof3}

We  note that there is a natural obstruction for noncompact  locally compact groups.

\begin{proposition}[Finite Invariant Volume and Compactness]\label{prop:finitehaar}
A locally compact group has finite nonzero  Haar measure if and only if it is compact. On a noncompact locally compact group there are infinitely many pairwise disjoint left translates of a fixed relatively compact identity neighborhood.
\end{proposition}
\begin{proof}
Choose an open relatively compact neighborhood $U$ with $0<\mu(U)<\infty$. If $G$ is noncompact, choose $g_n$ inductively outside the finite union $\bigcup_{k<n}g_k\overline{UU^{-1}}$. Then the sets $g_nU$ are disjoint, each of measure $\mu(U)$, so $\mu(G)=\infty$. In the compact case a Haar measure is finite on compact sets and can be normalized.
\end{proof}

The proof of our third main theorem can be presented now.

\begin{proof}[Proof of Theorem \ref{thm:tight}] Necessity follows by approximating the family by a finite $L^1$ net and choosing a common large ball for the finitely many integrable net functions. For sufficiency, the proof of Lemma~\ref{lem:tail}, with the Euclidean variable retained, gives the same uniform approximation by $E_n\mathcal F$. Averaging contracts Euclidean variation and does not increase the displayed tails.

For fixed $n$ these functions live on $\R^d\times K^n$. From the properties of the Euclidean translations and those of the Lie coordinate translations we see that Lemma~\ref{lem:translate} is satisfied. Convolutions with a smooth compactly supported approximate the identity therefore approximates the family uniformly in $L^1$. At each fixed scale the smoothed functions have uniformly bounded derivatives and are precompact on every compact set by Arzel\`a--Ascoli Theorem. Their Euclidean tails remain uniformly small: convolution with a kernel supported in a fixed Euclidean ball only enlarges the radius in the tightness estimate. This proves global $L^1$ precompactness at a fixed scale, then for $E_n\mathcal F$, and finally for $\mathcal F$ by the uniform tail coordinate estimate.
\end{proof}

\begin{example}[Escape by Group Translations]\label{ex:escape}
Let $B\subset\R^d$ be a bounded smooth domain of positive volume, and choose $t_k$ so that $B+t_k$ are pairwise disjoint. The sets
\[
\Omega_k=(B+t_k)\times K^{\N}
\]
have the same positive finite Haar measure and the same finite perimeter $\mathcal W_a(\one_{\Omega_k})=\mathcal P_{\R^d}(B)$. Their characteristic functions have pairwise $L^1$ distance $2|B|$. Thus global compactness cannot be deduced from volume and perimeter bounds on this noncompact locally compact group.
\end{example}

Theorem~\ref{thm:tight} is consistent with general Haar measure compactness criteria involving translations and tails \cite[Theorem 5.2]{krukowski}. It does not assert existence of isoperimetric minimizers on every noncompact group: a minimizing sequence need not be tight. 

\section{A relation with Carnot groups and sub-Finsler geometry}\label{subfinsler}

We compare the preceding results with some results on metric Lie groups, which are available in \cite{ledonne}. The comparison concerns the actual hypotheses and conclusions of that source. It separates established finite   dimensional geometry from the analytic estimates needed for the cartesian product, which supports the novelty of Theorems \ref{thm:intro} and \ref{thm:dual} in the literature on geometric measure theory on Lie groups. First of all, we note that a  \textit{sub-Finsler Lie group} in \cite[Definition 7.1.6]{ledonne} includes the full   tangent case. Its norms need not be smooth or strictly convex; see also \cite[Definitions 3.2.6, 3.2.7 and Remark 3.2.9]{ledonne}. Thus the following elementary identification places our finite   stage geometry within that framework.

\begin{proposition}[Finsler Identification in Finite Dimension]\label{prop:finsler}
For $n\geq1$ and a nontrivial compact connected Lie group $K$, define a norm on $\gk^n$ by
\[
F_{a,n}(v_1,\ldots,v_n)=\max_{1\leq j\leq n}\frac{|v_j|}{a_j}.
\]
Its left   invariant extension to $K^n$ is bi-invariant, and its dual norm is
\[
F_{a,n}^*(\xi_1,\ldots,\xi_n)=\sum_{j=1}^n a_j|\xi_j|.
\]
For $u\in L^1(K^n,m_K^{\otimes n})$, its divergence variation with this tangent norm is exactly
\begin{equation}\label{eq:finslerbv}
\sup_{\substack{Y\text{ smooth on }K^n \\
\\
F_{a,n}(Y)\leq1}
}
\left|\int_{K^n}u\,\diver Y\,dm_K^{\otimes n}\right|
=\sum_{j=1}^n a_j\mathcal V_j(u),
\end{equation}
with $\mathcal V_j(u)$  as per Definition \ref{def:variation}. The associated length distance is
\begin{equation}\label{eq:finslerdistance}
d_{a,n}(x,y)=\max_{1\leq j\leq n}\frac{d_K(x_j,y_j)}{a_j}.
\end{equation}
Here divergence is taken with respect to normalized product Haar measure.
\end{proposition}
\begin{proof}
The unit ball is the product of the balls $|v_j|\leq a_j$. Maximizing a linear functional separately on each ball proves the dual-norm formula. Orthogonality of the adjoint action gives bi-invariance. Write a smooth vector field in the invariant frame as $Y=(Y_1,\ldots,Y_n)$. The constraint in \eqref{eq:finslerbv} is exactly $|Y_j|\leq a_j$ for every $j$, and \[\diver Y=\sum^n_{j=1}\diver_jY_j.\] Independent optimization of these finitely many blocks, choosing their signs consistently, gives \eqref{eq:finslerbv}, also when a variation is infinite. For cylinder lifts to $G$, the averaging  in the unused coordinates shows that the finite   stage and product definitions agree. Every joining curve has length at least $d_K(x_j,y_j)/a_j$ for each $j$. Conversely, run minimizing geodesics in the $n$ factors on the same unit time interval, each at constant speed. Their product has length equal to the maximum in \eqref{eq:finslerdistance}.
\end{proof}

For smooth $u$, the integrand is consequently \[F_{a,n}^*(du)=\sum^n_{j=1} a_j|\nabla_j u|.\] The tangent norm is the weighted maximum, whereas the gradient norm is its dual weighted sum. This distinction is necessary for the identification. Haar measure remains specified: we make no identification with a differently normalized Finsler volume or Hausdorff measure. Completeness and geodesicity of these finite   stage distances are also direct applications of \cite[Theorem 7.1.7]{ledonne}; equivalence with other full   tangent invariant norms is \cite[Proposition 7.1.8]{ledonne}.

\begin{remark}It should be noted that \cite[Theorem 6.2.10]{ledonne} assumes finite  dimension to obtain a smooth homogeneous manifold. Berestovskii's Theorem, \cite[Theorem 7.4.1]{ledonne}, starts with a Lie coset manifold and an invariant geodesic distance inducing its topology. These finite dimensional hypotheses do not hold for $G$ in the assumptions of Theorems \ref{thm:intro} and \ref{thm:dual}.
\end{remark}

In fact,  the following result  help to identify  some obstructions between  the theory of infinite dimensional Lie groups and that  of Carnot groups. 

\begin{proposition}[Two Obstructions to a Carnot Interpretation]\label{prop:carnotobstruction}
For a nontrivial compact connected Lie group $K$, the compact group $K^{\N}$ has infinite  dimension. Moreover no compatible metric on $K^{\N}$ admits a surjective similarity with ratio different from one.
\end{proposition}
\begin{proof}
Put $r=\dim K\geq1$. A coordinate chart contains a closed embedded $r$-cube. Fixing all coordinates after $n$ at the identity embeds a closed $rn$-cube in $K^{\N}$. Monotonicity of the dimension for closed subsets of metrizable spaces, and the dimension of a cube, imply $\dim K^{\N}\geq rn$ for every $n$; see \cite{willard}. For a compatible metric let $D$ be the diameter. Compactness and nontriviality give $0<D<\infty$. A surjective similarity of ratio $\lambda$ would imply $D=\lambda D$, hence $\lambda=1$.
\end{proof}

The metric characterization in \cite[Theorem 11.4.1]{ledonne} requires local compactness, geodesicity, isometric homogeneity \emph{and} a nontrivial metric dilation. The dilation hypothesis is therefore unavailable here. A nontrivial Carnot group as in \cite[Definition 11.1.1]{ledonne} is simply connected and nilpotent, and is noncompact. It has infinite total Haar measure. Replacing $K$ by such a group in our proof would remove the normalized coordinate Haar averages, so it is not a substitution allowed by Theorem~\ref{thm:intro}.

\begin{remark}There is also a distinction between our bilateral invariance and a proper horizontal distribution. If a subspace $V\subset\gk$ is $\Ad(K)$-invariant, then $[\gk,V]\subset V$, so $V$ is already a Lie subalgebra. If $V$ is bracket generating, it follows that $V=\gk$. Thus a proper bracket-generating distribution cannot retain the same adjoint invariance used in Proposition~\ref{prop:equivariant}. One may study left   invariant sub-Riemannian structures on compact Lie groups, but their quotient and right    translation properties are peculiar. In fact,  $\SU(2)$ of Remark \ref{su2} uses the full tangent space.
\end{remark}

Note that \cite[Theorem 12.1.7]{ledonne} identifies  asymptotic cones of  simply connected nilpotent sub-Finsler Lie groups, while \cite[Theorem 12.1.8]{ledonne}  tangents of equiregular sub-Finsler manifolds.

\begin{remark}  \cite[Theorem 12.1.7]{ledonne} treats varying CC-bundle structures on a fixed manifold: a bracket-generating limiting distribution and a boundedly compact limiting distance give uniform convergence of the distances on compact sets. Unfortunately, none of these results supplies convergence or compactness of the weighted distributional variations on products with an increasing number of factors. In particular, metric convergence alone does not give the uniform $L^1$ estimate in Lemma~\ref{lem:tail}. However, to the best of our knowledge, an extension of \cite[Theorems 12.1.7 and 12.1.8]{ledonne} to the case of infinite dimensional Lie groups seems to be a corner of the literature which has not been investigated yet.
\end{remark}

\appendix

\section{Some  facts on topological groups  and measure theory}
\label{app:preserved-background}

We collect the classical facts underlying the constructions above. They also
clarify the distinction between the topology of a compact group, an invariant
distance, and a separately specified variation functional.

\subsection{Uniformities and invariant metrics}

A uniformity on a set $X$ is a nonempty family $\mathcal U$ of subsets of
$X\times X$, called entourages, which is closed under finite intersections
and passage to supersets, contains the diagonal in each of its members,
satisfies $U^{-1}\in\mathcal U$ for every $U\in\mathcal U$, and has the
property that for each $U\in\mathcal U$ some $V\in\mathcal U$ satisfies
$V\circ V\subseteq U$. Here
\[
 U^{-1}=\{(y,x):(x,y)\in U\},\qquad
 V\circ V=\{(x,z):\text{some }y\text{ satisfies }(x,y),(y,z)\in V\}.
\]
The sets $U[x]=\{y:(x,y)\in U\}$ form a neighbourhood basis at $x$.
The induced topology is Hausdorff precisely when
$\bigcap_{U\in\mathcal U}U=\{(x,x):x\in X\}$. A separated uniformity is
induced by a metric precisely when it has a countable basis. A Hausdorff
space admits a compatible uniformity precisely when it is completely
regular; in particular, every compact Hausdorff space does so
\cite{willard}.

For a Hausdorff topological group $H$, a canonical left uniformity has basic
entourages
\[
 U_V=\{(x,y):x^{-1}y\in V\},
\]
where $V$ runs over the identity neighbourhoods. Left translations are
homeomorphisms; they are not, in general, group homomorphisms. The
Birkhoff--Kakutani Metrization Theorem says that $H$ has a compatible
left   invariant metric if and only if its identity has a countable
neighbourhood basis \cite{hofmor}. Such a metric can be written
\[
 d(x,y)=\ell(x^{-1}y),
\]
where $\ell(e)=0$, $\ell(g)>0$ for $g\ne e$,
$\ell(g^{-1})=\ell(g)$, and $\ell(gh)\leq\ell(g)+\ell(h)$; the sets
$\{g:\ell(g)<r\}$ must form a neighbourhood basis. This is a group length,
not a Banach algebra norm. If $H$ is locally compact, metrizability is also
equivalent to the identity being a $G_\delta$ set.

For clarity, complete regularity and metrizability do not coincide. The
countable cube $[0,1]^{\mathbb N}$ is compact and metrizable, whereas an
uncountable cube $[0,1]^I$ is compact Hausdorff and nonmetrizable. Indeed,
a countable collection of basic neighbourhoods of the origin restricts only
countably many coordinates and cannot be a neighbourhood basis when $I$
is uncountable. The usual Schwartz space $\mathcal S(\mathbb R)$, with
seminorms
\[
 p_{k,l}(f)=\sup_{x\in\mathbb R}(1+|x|)^k|f^{(l)}(x)|,
 \qquad k,l\geq0,
\]
is a metrizable, completely regular Fr\'echet space. These examples emphasize
that an infinite dimensional topology may be metrizable without being the
topology of a finite   dimensional manifold.

\subsection{Finite dimensional quotients and finite support tori}

The torus group $\mathbb T$ is isomorphic to the circle group
$\mathbb S^1$, see \cite[Chapter 1]{hofmor}; its cartesian product
is a compact connected abelian  group, with quotient maps
\[
 \mathbb T^{\mathbb N}\longrightarrow\mathbb T^n,
 \qquad (x_j)_{j\geq1}\longmapsto(x_1,\ldots,x_n).
\]
Each $\mathbb T^n$ is also embedded as the subgroup with zero coordinates
after $n$. These are distinct operations: the projection forgets
coordinates, whereas the embedding specifies their values.

\begin{proposition}[Finite   Support Subgroups]\label{prop:background-direct-sum}
The subgroup
\[
 D=\mathbb T^{(\mathbb N)}
 =\{x\in\mathbb T^{\mathbb N}:x_j=0\text{ for all but finitely many }j \in \N\}
\]
is dense, proper, and Haar null in $\mathbb T^{\mathbb N}$. With its subspace
topology, the homomorphism 
\[s : x \in D \longmapsto s(x)=\sum^\infty_{j=1} x_j \in \T\] is discontinuous, although its restriction to every finite
torus is continuous.
\end{proposition}

\begin{proof}Every nonempty basic open set restricts only finitely many coordinates and therefore contains a point of $D$. The constant sequence with value
$1/2+\mathbb Z$ shows that $D$ is proper. Write
$D=\bigcup_nD_n$, where $D_n=\{x:x_j=0\ (j>n)\}$. Each $D_n$ is contained
in the Haar-null set $\{x:x_{n+1}=0\}$, so $D$ is Haar null.
Fix $t\ne0$ in $\mathbb T$ and let $x^{(j)}$ have value $t$ in coordinate
$j$ and zero elsewhere. Then $x^{(j)}\to0$ in the subspace topology but
$s(x^{(j)})=t$ for every $j$.
\end{proof}

In particular, the subspace topology on $D$ must not be identified with
the final topology for the inclusions
$\mathbb T^n\hookrightarrow\mathbb T^{n+1}$ in the category of topological
spaces: the latter makes $s$
continuous by its defining universal property. This is one reason why
direct limits and projective limits require separate treatment \cite{G1}.

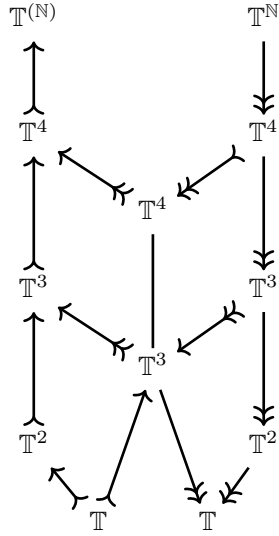
\begin{wrapfigure}{l}{0.38\textwidth}\label{butterfly}
\begin{minipage}{0.38\textwidth}
\centering
\begin{tikzpicture}[description/.style={fill=white,inner sep=2pt}]
\matrix (m) [matrix of math nodes, row sep=1.4em,
column sep=0.5em, text height=1.5ex, text depth=0.5ex]
{ \T^{(\N)} & & & & \T^\N \\
&&&&\\
\T^4 & & & &  \T^4  \\
&& \T^4 && \\ 
\T^3 &&&& \T^3 \\
&& \T^3 && \\
\T^2 &&&& \T^2 \\
&\T && \T & \\};

\path[->] (m-1-1) edge [<-<,line width=1pt] (m-3-1)
		 (m-3-1) edge [<-<<,line width=1pt] (m-4-3)
		         edge [<-<,line width=1pt](m-5-1)
		 (m-5-1) edge [<-<<,line width=1pt] (m-6-3)
		         edge [<-<,line width=1pt] (m-7-1)
		 (m-7-1) edge [<-<,line width=1pt] (m-8-2)
		 (m-8-2) edge [>->,line width=1pt]  (m-6-3)
		 (m-6-3) edge [->>,line width=1pt] (m-8-4)
		         edge [<-<<,line width=1pt] (m-5-5)
		         edge [-,line width=1pt](m-4-3)
		 (m-7-5) edge [<<-,line width=1pt] (m-5-5)
		 		 edge [->>,line width=1pt] (m-8-4)
		 (m-1-5) edge [->>,line width=1pt] (m-3-5)
		 (m-3-5) edge [>->>,line width=1pt] (m-4-3)
		 		 edge [->>,line width=1pt](m-5-5);
\end{tikzpicture}
\end{minipage}
\vspace{-10pt} 
\end{wrapfigure}

A group has \emph{no small subgroups} if some identity neighbourhood
contains no nontrivial subgroup. For compact groups, this is equivalent
to being a finite   dimensional Lie group \cite{hofmor}. Every identity
neighbourhood in $\mathbb T^{\mathbb N}$ contains a full coordinate circle
outside a finite set of coordinates, so this group fails that condition.
In fact it is not locally contractible. If a neighbourhood $U$ of the
identity were contractible within $\mathbb T^{\mathbb N}$, choose a
coordinate circle contained in a basic neighbourhood inside $U$.
Composing its inclusion and the alleged contraction with projection onto
that coordinate would contract the identity map of the circle, which is
impossible. Thus this compact pro-Lie group is not a finite   dimensional or
Hilbert manifold. By contrast, every compact abelian Lie group has the
form $\mathbb T^n\times F$ with $F$ finite abelian, and carries an ordinary
Riemannian structure \cite{hofmor,hilgertneeb}.

\begin{remark} Looking at $\T^\N$, any time we stop at a given index $j_0 \in \N$ in this projective limit, we have a compact abelian Lie group $\T^{j_0}$ which is known to possess the structure of Riemannian manifold, namely that of product of  $j_0$-copies of $\mathbb{S}^1$ with the usual Riemannian metric of the circle. In other words, if we fix $j_0=4$ and ask $j=4$ for all $j\in \N$ such that $j \ge 4$, then 
\begin{equation}\label{diag7}
\begin{tikzpicture}[node distance=2cm, auto] 
\node (A) {$\T^\N$};
\node (B) [below of=A] {$\T^4$};
\draw[->>] (A) to node [] {} (B);
\node (C) [left of=B] {$\T^3$};
\draw[->>] (A) to node [] {} (C);
\node (E) [left of=C] {$\T^2$};
\draw[->>] (A) to node [swap] {} (E);
\node (D) [left of=E] {$\T$};
\draw[->>] (A) to node [swap] {} (D);
\draw[->>] (B) to node [swap] {} (C);
\draw[->>] (C) to node [swap] {} (E);
\draw[->>] (E) to node [swap] {} (D);
\end{tikzpicture}
\end{equation}
which induces naturally surjections on each factor, that is, 
\begin{equation}\label{diag8}
\begin{tikzpicture}[node distance=2cm, auto] 
\node (B) [below of=A] {$\T^4$};
\node (C) [left of=B] {$\T^3$};
\node (E) [left of=C] {$\T^2$};
\node (D) [left of=E] {$\T$};
\draw[->>] (B) to node [swap] {} (C);
\draw[->>] (C) to node [swap] {} (E);
\draw[->>] (E) to node [swap] {} (D);
\end{tikzpicture}
\end{equation}
Here we keep the structure of compact abelian Lie group, and in turn that of Riemannian manifold of dimension four. Here both the structure of compact Lie group and Riemannian manifolds are compatible in dimension four. On the other hand, we have also a construction of direct limit tori of the following type
\begin{equation}\label{diag9}
\begin{tikzpicture}[node distance=2cm, auto] 
\node (A) {$\T^{(\N)}$};
\node (B) [below of=A] {$\T^4$};
\draw[<-<] (A) to node [] {} (B);
\node (C) [left of=B] {$\T^3$};
\draw[<-<] (A) to node [] {} (C);
\node (E) [left of=C] {$\T^2$};
\draw[<-<] (A) to node [swap] {} (E);
\node (D) [left of=E] {$\T$};
\draw[<-<] (A) to node [swap] {} (D);
\draw[<-<] (B) to node [swap] {} (C);
\draw[<-<] (C) to node [swap] {} (E);
\draw[<-<] (E) to node [swap] {} (D);
\end{tikzpicture}
\end{equation}
which induces naturally injections of the following form 
\begin{equation}\label{diag10}
\begin{tikzpicture}[node distance=2cm, auto] 
\node (B) [below of=A] {$\T^4$};
\node (C) [left of=B] {$\T^3$};
\node (E) [left of=C] {$\T^2$};
\node (D) [left of=E] {$\T$};
\draw[<-<] (B) to node [swap] {} (C);
\draw[<-<] (C) to node [swap] {} (E);
\draw[<-<] (E) to node [swap] {} (D);
\end{tikzpicture}
\end{equation}

and also in this situation it is possible to endow $\T^4$ with the structure of compact abelian Lie group and of Riemannian manifold (of dimension four).
We can get what is called a ``butterfly diagram'' (see the diagram on the top-left in this page) if we want to summarize what we have just noted. In the ``left wing''  we find \eqref{diag9} without the middle injections from $\T$ to $\T^{(\N)}$, from  $\T^2$ to $\T^{(\N)}$ and from $\T^3$ to $\T^{(\N)}$. The same observation can be done looking at the ``right  wing'' , where surjections are now present. The ``central part'' presents bijections at the level of $\T^4$ and $\T^3$ but also that we may omit the term $\T^2$ without a significant lost for the whole diagram. This reflects the fact that we may construct ``long wings'' in a similar diagram, stopping at a given term $\T^{j_0}$ and producing again a chain of injections and surjections on the external parts of the diagram, but keeping possible jumps in the central part. The product torus and quotient questions for topological groups also appear in \cite{morris2019,morris2020}.
 \end{remark}

\subsection{Subgroup measure and nonmeasurability} Another useful fact to recall is the role of nonmeasurability in compact groups. The following result can be found in \cite{hhm}, but we provide an alternative argument since we believe that it is instructive in our present context of investigation.

\begin{proposition}[Measurable Subgroups]\label{prop:background-subgroups} \

Let $G$ be a compact group with normalized Haar measure $m$. 
\begin{enumerate}
\item[{\rm (i).}] If a subgroup $H$ is measurable, then either $m(H)=0$, or $H$ is
open of finite index and $m(H)=[G:H]^{-1}$.
\item[{\rm (ii).}] A subgroup of countably infinite index is nonmeasurable. A subgroup
of finite index which is not closed is also nonmeasurable.
\item[{\rm (iii).}] If a subgroup is nonmeasurable, its closure is open and has finite
index.
\end{enumerate}
\end{proposition}

\begin{proof}
Translations act continuously on $L^1(G,m)$: first verify this for
continuous functions using uniform continuity on the compact group, then
use their density in $L^1$ and invariance of $m$. If a measurable $H$ has
positive measure, then
$\|\mathbf1_{gH}-\mathbf1_H\|_1<2m(H)$ for all $g$ sufficiently near the
identity. Hence $gH\cap H$ is nonempty, which implies $g\in H$.
Thus $H$ is open. Its cosets form an open cover of the compact space $G$,
so the index is finite, and invariance gives the stated measure. Every
open subgroup is closed because its complement is a union of open cosets.

If $G$ is the disjoint union of countably infinitely many cosets of a
measurable $H$, countable additivity and equality of their measures give
a contradiction, whether $m(H)$ is zero or positive. A measurable
finite   index subgroup has positive measure and is therefore closed by
the first part. Finally, if $\overline H$ is not open, the first part
gives $m(\overline H)=0$. Completeness of the measure then makes every
subset of $\overline H$, including $H$, measurable. This proves the last
assertion.
\end{proof}

These standard facts \cite{hhm} support the following explicit version of
the nonmeasurability example. It concerns non-Borel subgroups and is
independent of the perimeter compactness argument.

\begin{example}[A Nonmeasurable Subgroup of the Cartesian Product of  Tori]
Choose a Hamel basis of $\mathbb R$ over $\mathbb Q$ containing $1$ and an
irrational number $b$. Define a $\mathbb Q$-linear map
$\lambda:\mathbb R\to\mathbb Q$ by $\lambda(b)=1$ and by setting
$\lambda$ equal to zero on all other basis vectors. Since
$\mathbb Z\subseteq\ker\lambda$, it induces a surjective group
homomorphism
\[
 \overline\lambda:\mathbb T\to\mathbb Q,
 \qquad t+\mathbb Z\longmapsto\lambda(t).
\]
The subgroup $H=\ker\overline\lambda$ has countably infinite index and is
therefore nonmeasurable by Proposition~\ref{prop:background-subgroups}.
It is dense, since it contains $\mathbb Q/\mathbb Z$. Its inverse image
under the first-coordinate projection
$\mathbb T^{\mathbb N}\to\mathbb T$ also has countably infinite index and
is nonmeasurable. Neither subgroup is Borel. 
\end{example}

\subsection{Further classical comparisons}

The separation and mapping results below place the preceding examples in their usual topological context; none is an additional hypothesis of the compactness theorem. A completely regular Hausdorff space is called Tychonoff and embeds in a  cube $[0,1]^I$ of dimension $|I|$. Every locally compact Hausdorff space is Tychonoff. Every compact Hausdorff space is normal, and every metrizable space is normal. See terminology and definitions in \cite{willard}. Thus the Tychonoff plank \cite[18 G]{willard}, being compact Hausdorff, is completely regular; its subspaces remain completely regular. The Sorgenfrey line \cite[Examples 15.3]{willard} is regular, since its basic half-open intervals are clopen. The Moore plane \cite[Example 15.3,  Problem 4B]{willard} provides a completely regular space which is not normal. These examples distinguish separation axioms, rather than provide obstructions to Haar averaging for compact groups and locally compact groups in Theorems \ref{thm:intro} and\ref{thm:dual}, or in Lemmas \ref{thm:tower}.

For a normal Hausdorff space, Urysohn's Lemma (see \cite{willard}) separates disjoint closed sets by a continuous function into $[0,1]$, and Tietze's theorem extends a continuous real-valued function on a closed subset to the whole space. These are useful classical extension principles, but carry no canonical group action or invariant measure. Therefore they do not offer a significant framework of generalization to our main results.

The Cantor cube $\{0,1\}^{I}$ is a compact group with coordinatewise addition modulo two. A space is dyadic if it is a continuous image of a Cantor cube, see \cite[Alexandroff Theorem on Dyadic Spaces,  Lemma A4.31, Corollary A4.32, Proposition A4.20]{hofmor}. Every nonempty compact metrizable space is a continuous image of $\{0,1\}^{\N}$; a nonempty compact metrizable space is homeomorphic to that countable cube precisely when it is totally disconnected and has no isolated points. Products of nonempty compact metrizable spaces are consequently dyadic, by taking products of the surjections. A nonempty compact connected locally connected metrizable space is a Peano space; the Hahn--Mazurkiewicz Theorem  \cite[Theorem 31.5]{willard} characterizes these as the Hausdorff continuous images of $[0,1]$. Such continuous images need not be homomorphisms, so these topological statements do not replace the quotient results of Section~\ref{isoperimetricomparison}.

Finally, in the finite   dimensional linear context, the continuous linear endomorphisms of $\R^n$ form $M_n(\R)$, a normed algebra, and its open group of units $\mathrm{GL}_n(\R)$ has dimension $n^2$. More generally the pointwise operator topology on continuous linear maps must be distinguished from the operator norm topology in infinite dimension. In topological group theory, a group is \textit{almost connected} when $G/G_0$ is compact with $G_0$ largest connected normal subgroup of $G$. The general pro-Lie definition does not require local compactness, but almost connected group turn out to be examples of pro-Lie groups; for example an infinite product of copies of $\R$ is pro-Lie but is not locally compact \cite{prolie,hilgertneeb}. Therefore the category of pro-Lie groups with corresponding morphisms is extremely rich and sophisticated.



\end{document}